\documentclass[a4paper,12pt]{amsart}

\usepackage{amssymb,amsmath}
\usepackage[dvips]{graphics}
\usepackage[dvips]{color}
\usepackage{graphicx}
\usepackage[english]{babel}
\usepackage[latin1]{inputenc}
\usepackage{mathrsfs}   
\usepackage{comment}
\usepackage[textsize=tiny]{todonotes}

\newtheorem{theorem}{Theorem}[section]
\newtheorem{lemma}[theorem]{Lemma}
\newtheorem{proposition}[theorem]{Proposition}

\theoremstyle{definition}
\newtheorem{definition}[theorem]{Definition}
\newtheorem{example}[theorem]{Example}
\newtheorem{remark}[theorem]{Remark}

\numberwithin{equation}{section}

\author{Laurent Moonens}
\author{Emmanuel Russ}
\author{Heli Tuominen}

\newcommand\rn{\mathbb R^n}
\newcommand\re{\mathbb R}
\newcommand\n{\mathbb N}

\newcommand\ph{\varphi}

\newcommand\cM{\operatorname{\mathcal M}}
\newcommand\cB{\mathcal B}

\newcommand\cA{\mathcal A}
\newcommand\cG{\mathcal G}

\newcommand\dist{\operatorname{dist}}
\newcommand\grad{\nabla}

\providecommand{\ch}[1]{\text{\raise 2pt \hbox{$\chi$}\kern-0.2pt}_{#1}}

\providecommand{\vint}[1]{\mathchoice
          {\mathop{\vrule width 5pt height 3 pt depth -2.5pt
                  \kern -9.5pt \kern 1pt\intop}\nolimits_{\kern -5pt{#1}}}%
          {\mathop{\vrule width 5pt height 3 pt depth -2.6pt
                  \kern -6pt \intop}\nolimits_{\kern -3pt{#1}}}%
          {\mathop{\vrule width 5pt height 3 pt depth -2.6pt
                  \kern -6pt \intop}\nolimits_{\kern -3pt{#1}}}%
          {\mathop{\vrule width 5pt height 3 pt depth -2.6pt
                  \kern -6pt \intop}\nolimits_{\kern -3pt{#1}}}}

\newcommand{\N}{\mathbb{N}}     
\newcommand{\R}{\mathbb{R}}     

\newcommand{\calE}{\mathscr{E}}

\newcommand{\calH}{\mathscr{H}}

\DeclareMathOperator{\Capa}{Cap}			
\DeclareMathOperator{\capa}{cap}			
\DeclareMathOperator{\diver}{div}			
\DeclareMathOperator{\Lip}{Lip}				
\DeclareMathOperator{\capp}{cap}				
\DeclareMathOperator{\supp}{supp}			
\DeclareMathOperator{\loc}{loc}			

\usepackage[mathcal]{eucal}					

\newcommand{\la}{\langle}
\newcommand{\ra}{\rangle}

\renewcommand{\geq}{\geqslant}
\renewcommand{\leq}{\leqslant}
\renewcommand{\epsilon}{\varepsilon}

\newcommand{\liminfe}{\mathop{\underline{\lim}}}

\newcommand{\supess}{\mathop{\mathrm{ess\,sup}}}
\newcommand{\infess}{\mathop{\mathrm{ess\,inf}}}

\usepackage{color}

\begin{document}

\title[Singularities for $\diver v=f$ in weighted spaces]{Removable singularities for $\diver v=f$\\in weighted Lebesgue spaces}

\maketitle

\begin{abstract}
Let $w\in L^1_{loc}(\R^n)$ be a	positive weight. Assuming that a doubling condition and an $L^1$ Poincar\'e inequality on balls for the measure $w(x)dx$, as well as a growth condition on $w$, we prove that the compact subsets of $\R^n$ which are removable for the distributional divergence in $L^{\infty}_{1/w}$ are exactly those with vanishing weighted Hausdorff measure. We also give such a characterization for $L^p_{1/w}$, $1<p<+\infty$, in terms of capacity. This generalizes results due to Phuc and Torres, Silhavy and the first author. 
\end{abstract}

\section{Introduction}
In the past years, \emph{removable singularities} of bounded vector fields satisfying $\diver v=0$ in the distributional sense have been studied, \emph{e.g.} by the first author \cite{M2006}, Silhavy \cite{SILHAVY} and Phuc and Torres \cite{PT}. It has been shown, in particular, that a compact set $S\subseteq\R^n$ can contain a non void support of the distributional divergence of a bounded vector field on $\R^n$, if and only if its $(n-1)$-dimensional Hausdorff measure is positive. 
As a matter of fact, all those results have immediate counterparts for vector fields defined on an open subset $\Omega$ of $\R^n$, satisfying the equation $\diver v=f$, where $f$ is a locally integrable function on $\Omega$, in case the latter equation admits at least one solution in $L^\infty(\Omega)$.

Given $n/(n-1)<p<\infty$, Phuc and Torres in \cite{PT} showed a corresponding result for $L^p$-vector fields. More precisely, given an open set $\Omega\subseteq\R^n$ and a locally integrable function $f$ in $\Omega$ for which the equation $\diver v=f$ is solvable in $L^p(\Omega)$, their results imply that a compact set $S\subseteq \Omega$ contains a non void support of the distributional divergence of an $L^p$-vector field in $\R^n$, if and only if $\capp_{p'}(S)>0$, where $\capp_{p'}$ is the capacity associated to the Sobolev space $W^{1,p'}(\R^n)$ (see Definition~\ref{Definition: Sobo capacity} below) and $p'$ is the conjugate exponent to $p$ verifying $1/p +1/p'=1$.

On the other hand, given a (bounded) domain $\Omega\subset\R^n$, it may happen that it is \emph{not} possible to find a constant $C>0$ such that given any $f\in L^\infty(\Omega)$, the equation
$$
\diver v=f
$$
admits a bounded solution $v\in L^\infty(\Omega,\R^n)$ satisfying $\|v\|_\infty\leq C\|f\|_\infty$. 
In fact, the existence in this context of an integrable weight $w>0$ such that the divergence operator 
acting from the weighted Lebesgue space $L^\infty_{1/w}(\Omega,\R^n)$ to the usual space $L^\infty(\Omega)$, admits a continuous right inverse, has been shown by Duran, Muschietti, the second author and Tchamitchian in \cite{DMRT} to be \emph{equivalent} to the integrability of the geodesic distance (in $\Omega$) to a fixed point $x_0\in\Omega$. Under the latter integrability property, a similar invertibility result also holds when $L^\infty_{1/w}(\Omega,\R^n)$ and $L^\infty(\Omega)$ are replaced by $L^p_{1/w}(\Omega,\R^n)$ and $L^p(\Omega)$, respectively, with $1<p<\infty$.

In order to get some understanding of how the introduction of a (locally) integrable weight $w$ influences, in the associated weighted Lebesgue spaces, the set of singularities of a vector field having a prescribed divergence, we shall assume here that $\Omega=\R^n$ and $f=0$ in the sequel (leaving the case where $\Omega$ is a bounded domain for a future work), and study first (see section~\ref{sec.linfini}) the possible sets of singularities of vector fields in $L^\infty_{1/w}(\R^n,\R^n)$ solving $\diver v=0$ (or $\diver v=f$ for some locally integrable function $f$ on $\R^n$ yielding at least a solution in $L^\infty_{1/w}(\R^n,\R^n)$), and make a similar study in appropriate weighted $L^p$ spaces.

More precisely, calling \emph{$L^p_{1/w}$-removable} any compact subset of $\R^n$ that does \emph{not} support any nonzero distributional divergence of a vector field $v\in L^p_{1/w}(\R^n,\R^n)$, we show the following result (which combines our Theorems~\ref{thm.iff infty} and \ref{rem equiv cap}).
\begin{theorem}\label{thm.fond}
Assume $S$ is a compact subset of $\R^n$.
\begin{itemize}
\item[(i)] If $p=\infty$ and if the weight $w$ is $1$-admissible and satisfies the growth condition (\ref{eq.croiss}) below, then $S$ is $L^\infty_{1/w}$-removable if and only if $\calH^h(S)=0$, where $\calH^h$ is the Hausdorff outer measure associated to $w$ as in Section~\ref{subsection.haus}.
\item[(ii)] If $1<p<\infty$ and if one has $w^{p'-1}\in A_{p'}$, then $S$ is $L^p_{1/w}$-removable if and only if one has $\Capa_{p'}^{w^{p'-1}}(S)=0$, where $\Capa_{p'}^{w^{p'-1}}$ is the Sobolev capacity associated to $w^{p'-1}$ appearing in Definition~\ref{Definition: Sobo capacity}.
\end{itemize}
\end{theorem}
\noindent In the previous statement, we mean by saying that a weight is \emph{$1$-admissible}, that it is doubling and satisfies a $(1,1)$-Poincar\'e inequality (see Definition~\ref{def.adm} below), while $A_{p'}$ stands for the Muckenhoupt class introduced (see Definition~\ref{def.muc}). Note that any $A_1$ weight is $1$-admissible (see Remark~\ref{rmq.adm}).  Note that $\calH^h$ is the classical (spherical) Hausdorff measure of dimension $n-1$ is case $w=1$.

A first remark about the previous theorem is that we recover, when $w=1$, the result mentioned above stating that $L^\infty$-removable (compact) sets are exactly those satisfying $\calH^{n-1}(S)=0$, and similarly in $L^p$.

An interesting case covered by our results is the one when the weight $w$ equals $+\infty$ on a ``large'' set (\emph{e.g.} on a set of positive Hausdorff dimension)~---~allowing the vector fields in $L^\infty_{1/w}$ to have singular pointwise behaviour on this ``large'' set. More precisely, as an interesting complement to [Theorem~\ref{thm.fond}, (i)], we provide examples of $A_1$ weights $w$ of the form $w(x):=\dist(x,C)^{-\alpha}$, with $\alpha>0$, whose singular set $C$ has positive Hausdorff dimension yet \emph{is} $L^\infty_{1/w}$-removable for the equation $\diver v=0$ for some values of $\alpha$ related to the Hausdorff dimension of $C$ and that of the ambient space. This is our Example~\ref{ex.cantor}.

The proof of both parts of Theorem~\ref{thm.fond} follow the same structure. Sufficient removability conditions are obtained by truncation arguments as in De Pauw \cite{DP} and \cite{M2006} (case $p=\infty$) and \cite{PT} (case $1<p<\infty$), but extra care is needed for we cannot, in this weighted context, rely on straightforward estimates relating the (weighted) perimeter of a ball $B$ of radius $r$ to the integral $\frac 1r \int_B w$. Showing those conditions are also necessary for a set $S$ to be removable is done by solving the equation $\diver v=\mu$ in $L^p_{1/w}(\R^n,\R^n)$ for some suitable measures $\mu$ supported in $S$, and showing it might admit non trivial solutions in case $S$ does \emph{not} satisfy the conditions in question. This is done, as in Bourgain and Brezis \cite{BBr} and \cite{PT}, by using a simple version of the closed range theorem. Note that the case where $p=+\infty$ cannot be dealt with using capacity arguments, see Remark \ref{remcapacity} below.

The paper is organized as follows. 
In Section \ref{sec.notation}, we give definitions and basic properties of $p$-admissible weights and introduce the notations used in the paper. 
In Section \ref{sec.linfini}, we study the removability question for the divergence equation for weighted $L^\infty$-vector fields. 
In the proofs, we need theory of functions of bounded variation in the weighted case, the weighted Hausdorff measure of co-dimension one, the boxing inequality and a version of Frostman's lemma. Those, as well as some technical lemmas are presented before the main results of the section.
Section \ref{sec.lp} contains characterization of removable sets for the divergence equation for weighted $L^p$-vector fields. In this section, important tools are weighted Sobolev spaces, different capacities and some tools from the general theory of $L^q$-capacities, discussed before the main results.

\section{Weights and notation}\label{sec.notation}

A locally integrable function $w\colon\rn\to\re$ is {\em a weight} if $w(x)>0$ for almost every $x\in\rn$. 
We say that the weight $w$ is \emph{doubling} if there exists a constant $C_D\geq 1$ (called the \emph{doubling constant of $w$}) such that for any $x\in\R^n$ and any $r>0$ one has:
$$
\int_{B(x,2r)} w\leq C_D \int_{B(x,r)} w,
$$
where $B(x,r)$ denotes the Euclidean (open) ball with center $x$ and radius $r$ in $\R^n$ and where one integrates with respect to Lebesgue measure. An iteration of the doubling inequality then ensures that one has, for all $t>0$:
\begin{equation}\label{eq.doubl}
\int_{B(x,tr)} w\leq C_D t^{s_D}\int_{B(x,r)} w,
\end{equation}
where $s_D:=\log_2 C_D$ is the \emph{doubling dimension of the weighted space $(\R^n,w)$}.

{\em For $1\leq p\leq\infty$, the weighted $L^p$-space}, $L^p_w(\R^n,\R^k)$, consists of measurable functions $u\colon\R^n\to\R^k$ for which $|u|^pw\in L^1(\R^n,\R^k)$, and we let:
$$
\|u\|_{p,w}:=\left(\int_{\R^n} |u|^pw\right)^{1/p},
$$
for $p<\infty$, and
$$
\|u\|_{\infty,w}:=\|uw\|_\infty.
$$
We finally let $L^p(\R^n):=L^p(\R^n,\R)$.\\

In the sequel we shall denote by $\Lip_c(\R^n)$ the set of all compactly supported (real valued) Lipschitz functions in $\R^n$. 
For a weight $w$, the weighted Euclidean space endowed with the Euclidean metric and the measure $d\mu=w\,dx$ is denoted by 
$(\R^n,w)$.

\begin{definition} Let $1\leq p<+\infty$ be a real number.
 We shall say that the weighted space $(\R^n,w)$ supports {\em a weighted $(1,p)$-Poincar\'e inequality} in case that there exist constants $C_P>0$ and $\tau>1$ such that for any $\varphi\in \Lip_{c}(\R^n)$, any $x\in\R^n$ and any $r>0$ we have:
$$
\vint{B(x,r)} |\varphi-\bar{\varphi}_{x,r}|w\leq C_P r\bigg(\,\,\vint{B(x,\tau r)} |\nabla \varphi|^pw\bigg)^{1/p},
$$
where we let $\bar{\varphi}_{x,r}:=\vint{B(x,r)} \varphi w$ and where $\vint{B}{\varphi}w$ denotes the mean value 
$\frac{1}{\int_B w} \int_B \varphi w$ for any Borel set $B$ satisfying $|B|>0$.
\end{definition}

We use the class of $p$-admissible weights as in \cite[Section~1.1]{HeKiMa} and \cite[Definition~A.6]{BB}.
Such weights are important in the nonlinear potential theory developed in \cite{HeKiMa}, see also \cite[Appendix A]{BB}.

\begin{definition}\label{def.adm}
Let $1\leq p<\infty$ be a real number.
A weight $w$ is said to be \emph{$p$-admissible} in case it is doubling and the weighted space $(\R^n,w)$ supports a $(1,p)$-Poincar\'e inequality.
\end{definition}

It follows from Heinonen, Kilpel\"ainen and Martio \cite[Corollary~20.9]{HeKiMa} (in the second edition of their book) that $p$-admissible weights for $p>1$ satisfy a bunch of other interesting properties, among which the following Poincar\'e inequality for compactly supported functions (see \cite[Section~1.4]{HeKiMa} and \cite[Corollary~2.1.5]{Tu}): there exists $\kappa>0$ such that for any ball $B=B(x,r)\subseteq\R^n$ and any $\varphi\in\Lip_c(\R^n)$ supported in $B(x,r)$, we get:
\begin{equation}\label{eq.poincare'}
\int_{B(x,r)} |\varphi|^p w\leq \kappa \,r^p\int_{B(x,r)} |\nabla\varphi|^pw.
\end{equation}

An important class of $p$-admissible weights are the $A_p$-weights, which were defined by Muckenhoupt in \cite{Mu}, where he showed that when $1<p<\infty$, the Hardy--Littlewood maximal operator is $L^p_w$-bounded if and only if $w\in A_p$. 

\begin{definition}\label{def.muc}
 A weight $w$ is {\em an $A_p$-weight}, $1<p<\infty$, if 
\[
\sup_B\Big(\,\vint{B}w\Big)\Big(\,\vint{B}w^{1/(1-p)}\Big)^{p-1}<\infty,
\]
where the supremum is taken over all balls $B\subset\rn$. Note that the $A_p$-condition implies that $w\in L^{1/(1-p)}(\rn)$ \emph{locally}.

A weight $w$ is {\em an $A_1$-weight}, if
\[
\sup_B\Big(\,\vint{B}w\Big)\supess_{B}\frac1{w}<\infty.
\]
\end{definition}
\begin{remark}\label{rmq.adm}
The fact that $A_p$-weights are $p$-admissible has been proved in \cite[Theorem~15.21]{HeKiMa} in case $p>1$ and in \cite[Theorem 4]{Bj} in case $p=1$. The doubling property follows easily from the $A_p$-condition but the validity of a weighted $(1,p)$-Poincar\'e inequality requires more work. 
\end{remark}

For further properties of $A_p$-weights, see for example \cite{CF}, \cite{J}, \cite[Chapter V]{St2} and \cite[Remark 1.2.4]{Tu} and for examples of $p$-admissible weights that are not $A_p$-weights, see for example \cite{CW} and \cite{FGW}.

\begin{example}\label{ex.a1}
It is classical that for $-n<\eta\leq 0$, $w_\eta(x):=|x|^\eta$ is an $A_1$-weight. Moreover any weight of the form $w_\eta(x)=|x|^\eta$ for $\eta>-n$ is doubling (while it may \emph{not} be $A_1$). See \emph{e.g.} \cite[Example~1.2.5]{Tu}.
\end{example}

For a set $A \subset\rn$, $\cM_+(A)$ is the set of locally finite (nonnegative) Radon measures supported in $A$.

\section{The case of weighted $L^\infty$ vector fields}\label{sec.linfini}
In this section, we study the removability question for the divergence equation for weighted $L^\infty$-vector fields. 
We start by defining some tools and proving results needed in the proofs - those include weighted Hausdorff content and measure of codimension $1$ and functions of bounded variation in the weighted setting. 
Our main results in this section hold for doubling weights that satisfy a $(1,1)$-Poincar\'e inequality. In Theorem \ref{thm.suff}, we show that any compact set $S\subseteq\R^n$ with $\calH^h(S)=0$, is $L^\infty_{1/w}$-removable for $\diver v=0$. Indeed, vanishing Hausdorff measure almost characterizes removable sets - if the weight $w$ satisfies an additional mild growth condition, then a compact set $S\subseteq\R^n$ is $L^\infty_{1/w}$-removable for $\diver v=0$ if and only if $\calH^h(S)=0$, see Theorem \ref{thm.iff infty}.

\subsection{Hausdorff contents}\label{subsection.haus}

Let $w$ be a weight.
Associated to $w$, define a (spherical) measure function $h$ on (closed) balls $B(x,r)$ by:
$$
h(B(x,r)):=\frac 1r \int_{B(x,r)} w.
$$

According to the usual Carath\'eodory construction (see \cite[Section~4.1]{MATTILA}), we also define a weighted co-dimension 1 (spherical) Hausdorff outer measure (as in Turesson \cite[Section~2.3]{Tu} and Nieminen \cite{Ni}) by letting first, for $0<\delta\leq\infty$ and $E\subseteq\R^n$:
$$
\calH^h_\delta(E):=\inf\sum_{j\in J} h(B(x_j,r_j)),
$$
where the infimum is taken on all countable coverings of $E$ by balls $(B(x_j,r_j))_{j\in J}$ satisfying $r_j\leq\delta$ for all $ j\in J$. We define then 
$$
\calH^h(E):=\lim_{\delta\to 0} \calH^h_\delta(E).
$$ 
It follows from \cite[Proposition~2.3.3]{Tu} that $\calH^h_\delta$ is an outer measure for any $0<\delta\leq\infty$, and that $\calH^h$ is a Borel regular outer measure.

The following straightforward lemma will be useful in the sequel.
\begin{lemma}\label{lem.leb-negl}
Let $h$ be associated to the weight $w$ as above and assume that the compact set $S\subseteq\R^n$ satisfies $\calH^h(S)=0$. 
Then $S$ is Lebesgue-negligible.
\end{lemma}

\begin{proof}
If $\eta>0$ is an arbitrary positive number, let $(B(x_j,r_j))_{j\in J}$ be a finite family of balls covering $S$, verifying $r_j\leq 1$ for each $j\in J$ as well as:
$$
\sum_{j\in J} h(B(x_j,r_j))\leq \eta.
$$
We then compute:
$$
\int_S w\leq\int_{\bigcup_{j\in J} B(x_j,r_j)} w\leq\sum_{j\in J}\int_{B(x_j,r_j)} w
= \sum_{j\in J} r_j h(B(x_j,r_j))\leq\sum_{j\in J}h(B(x_j,r_j))\leq\eta,
$$
so that one has $\int_S w=0$, since $\eta>0$ is arbitrary. 
It follows that the set $\{x\in S: w(x)>0\}$ is Lebesgue-negligible, and hence that $S$ itself is Lebesgue-negligible for we assumed than one has $w>0$ a.e. on $\R^n$.
\end{proof}

When $\R^n$ is endowed with a doubling weight which grows fast enough, we have the following version of Frostman's lemma, which is a particular case of \cite[Theorem~3.4.27]{Tu}.
\begin{lemma}\label{lem.frostman}
Assume that $w$ is a doubling weight on $\R^n$. If moreover, for any $x\in\R^n$:
\begin{equation}\label{eq.croiss}
\lim_{r\to\infty} h(B(x,r))=\infty,
\end{equation}
and if $B\subseteq\R^n$ is a Borel set verifying $\calH^h_\infty(B)>0$, then there exists a nontrivial measure $\mu\in\cM_+(B)$ satisfying the following inequality:
$$
\mu(B(x,r))\leq Ch(B(x,r)),
$$
for any $x\in\R^n$ and $r>0$.
\end{lemma}

\begin{example}
An easy computation shows that the weight $w_\eta$ defined in Example~\ref{ex.a1} also satisfies condition (\ref{eq.croiss}) in case $1-n<\eta\leq 0$.
\end{example}

In fact, assuming the growth condition (\ref{eq.croiss}) only for $x=0$, it is equivalent for a bounded set $B$ to satisfy $\calH^h_\infty(B)=0$ or $\calH^h(B)=0$.
\begin{lemma}\label{lem.equiv}
Assume that $w$ is doubling and that one has:
\begin{equation}\label{eq.croiss2}
\lim_{r\to\infty} h(B(0,r))=\infty.
\end{equation}
For any bounded set $B\subseteq \R^n$, the equalities $\calH^h(B)=0$ and $\calH^h_\infty(B)=0$ are equivalent.
\end{lemma}

\begin{proof}
Fix a bounded set $B\subseteq\R^n$. Since one has $\calH^h_\infty(B)\leq\calH^h(B)$, it is clear that $\calH^h_\infty(B)=0$ provided that $\calH^h(B)=0$. 

Conversely, assume that $\calH^h_\infty(B)=0$. Choose $R_0>0$ such that one has $B\subset B(0,R_0)$. Let $f(r):=h(B(0,r))$ for all $r>0$. Since (\ref{eq.croiss2}) yields $$\lim_{r\rightarrow +\infty} f(r)=+\infty,$$ we are allowed to choose $R>R_0$ such that $f(r)>3^{s_D}C_D$ for all $r>R$, where $C_D>0$ and $s_D:=\log_2 C_D$ are the doubling constants of $w$ (see (\ref{eq.doubl}) above).

Fix now $\delta>0$ and let:
$$
c_\delta:=\max\Big(3, 2+\frac{R_0}{\delta}\Big).
$$
Choose then $\varepsilon>0$ with:
$$
\varepsilon<\min\big(1,\frac 1{c_\delta^{s_D}C_DR}\int_{B(0,\delta)}w\bigg).
$$
Since one has $\calH^h_\infty(B)=0$, there exists an (at most countable) family of balls $(B(x_j,r_j))_{j\in J}$ covering $B$ and satisfying:
$$
\sum_{j\in J} h(B(x_j,r_j))=\sum_{j\in J} \frac 1{r_j}\int_{B(x_j,r_j)} w\leq \varepsilon.
$$
One may assume that, for all $j\in J$, one has $B(x_j,r_j)\cap B\neq \emptyset$, so that one computes $\left\vert x_j\right\vert\leq r_j+R_0$. By the doubling property (\ref{eq.doubl}), we get for any $j\in J$ such that $r_j>R_0$:
\begin{multline*}
\int_{B(0,r_j)}w\leq \int_{B(x_j,r_j+\left\vert x_j\right\vert)} w\leq C_D\left(1+\frac{\left\vert x_j\right\vert}{r_j}\right)^{s_D}\int_{B(x_j,r_j)}w\\\leq C_D\left(2+\frac{R_0}{r_j}\right)^{s_D}\int_{B(x_j,r_j)}w\leq 3^{s_D}C_D \int_{B(x_j,r_j)}w,
\end{multline*}
so that:
$$
\sum_{j\in J} f(r_j)\leq 3^{s_D}C_D\varepsilon < 3^{s_D}C_D.
$$
We hence get $r_j\leq R$ for all $j\in J$, for the latter inequality is obvious in case $j\in J$ is such that one has $r_j\leq R_0<R$.

Assume now that $j\in J$ is such that $r_j>\delta$. If moreover one has $r_j> R_0$, the computations above show that:
$$
\int_{B(0,r_j)}w\leq 3^{s_D} C_D r_j\cdot\frac{1}{r_j} \int_{B(x_j,r_j)} w\leq 3^{s_D}C_DR\varepsilon\leq c_\delta^{s_D} C_D R\epsilon.
$$
In case one has $\delta<r_j\leq R_0$, we compute using again (\ref{eq.doubl}):
$$
\int_{B(0,r_j)}w\leq C_D\left(2+\frac{R_0}{r_j}\right)^{s_D}  \int_{B(x_j,r_j)} w\leq \left(2+\frac{R_0}{\delta}\right)^{s_D}C_Dr_j\varepsilon\leq c_\delta^{s_D} C_D R\epsilon.
$$
Hence in both cases we have $\int_{B(0,r_j)} w\leq c^{s_D}_\delta C_DR\epsilon$. We hence get:
$$
\int_{B(0,\delta)} w\leq \int_{B(0,r_j)} w\leq c_\delta^{s_D} C_DR\epsilon<\int_{B(0,\delta)}w,
$$
which is impossible.

Therefore $r_j\leq \delta$ for each $j\in J$, so that we get:
$$
\calH^h_\delta(B)\leq\sum_{j\in J} h(B(x_j,r_j))\leq\epsilon.
$$
Since $\epsilon$ is arbitrary small, this yields $\calH^h_\delta(B)=0$. Finally, we get $\calH^h(B)=0$ for the previous estimates yield $\calH^h_\delta(B)=0$ for any $\delta>0$.
\end{proof}

\begin{remark}
In case $w$ is $1$-admissible, the previous lemma can be obtained by combing results \cite[Lemma~7.6, Remark~7.4]{KINNUNEN-ET-AL} by Kinnunen, Korte, Shanmugalingam and the third author.
\end{remark}

\subsection{Miranda's $BV$-functions} 
In the more general context of metric measure spaces, M.~Miranda introduced in \cite{MIRANDA} the notion of \emph{function with bounded variation}. We shall in the sequel particularize some results obtained by Miranda to the context where $\R^n$ is endowed by an appropriate weight; to this purpose, we need to introduce some terminology.

Following Miranda \cite{MIRANDA}, read in this weighted context by Camfield \cite{CAMFIELDT}, define the \emph{metric (weighted) variation of $u\in L^1_{w}(\R^n)$} by:
$$
\|Du\|_w:=\inf\left\{\liminfe_{k\to\infty} \int_{\R^n} |\nabla \varphi_k|w: (\varphi_k)\subseteq\Lip(\R^n), \varphi_k\to u\text{ in }L^1_{w}\right\}.
$$

The following Theorem, stated here for reader's convenience but unnecessary for our purposes, is a direct consequence of two deep results by Camfield \cite[Theorems~3.2.6 and 3.4.5]{CAMFIELDT}. It shows that, under some regularity conditions on the weight $w$, the metric variation of a Lipschitz function is identical to the $L^1_w$ norm of its gradient.
\begin{theorem}[Camfield]\label{thm.camfield}
Assume that $w$ is locally integrable and lower semicontinuous. If moreover there exists an $\calH^h$-negligible set $E$ outside which $w$ is continuous and (strictly) positive, then for any $\varphi\in\Lip_c(\R^n)$ we have:
\begin{equation}\label{eq.egal}
\|D\varphi\|_w=\|\nabla\varphi\|_{1,w}=\int_{\R^n} |\nabla\varphi|w.
\end{equation}
\end{theorem}

For $1$-admissible weights, the equality (\ref{eq.egal}) can be replaced by a comparison between the two quantities involved, at least for Lipschitz functions (see \cite[p.~992, (19) and below]{MIRANDA}).
\begin{proposition}\label{prop.ineg-normes}
Assume that the weight $w$ is $1$-admissible. 
Then for any $\varphi\in\Lip_c(\R^n)$ we have:
$$
c\|\nabla\varphi\|_{1,w}\leq \|D\varphi\|_w\leq \|\nabla \varphi\|_{1,w},
$$
where $c>0$ is independent of $\varphi$.
\end{proposition}

The following proposition is a particular case of M.~Miranda's Coarea formula \cite[Proposition~4.2]{MIRANDA}. The \emph{(weighted) perimeter} of a Borel set $B$, denoted by $P_w(B)$, is defined by:
$$
P_w(B):=\|D\chi_B\|_w.
$$
\begin{proposition}\label{prop.coaire}
Assume that $w$ is a $1$-admissible weight.
Then, for any $u\in L^1_{w}(\R^n)$ verifying $\|Du\|_w<+\infty$, we have:
$$
\|Du\|_w=\int_\R P_w(\{u>t\})\,dt.
$$
\end{proposition}

We shall also make use of the following boxing inequality, due (in the more general framework of measure metric spaces) to Kinnunen, Korte, Shanmugalingam and the third author \cite[Theorem~3.1]{KINNUNEN-ET-AL}.
\begin{theorem}[Boxing inequality]\label{thm.boxing}
Assume that $w$ is a $1$-admissible weight.
There exists a constant $C_B=C_B(C_D,C_P)>0$ such that for any open set $U\subseteq\R^n$ verifying $\int_Uw<\infty$, we can find sequences $(x_i)\subseteq U$ and $(r_i)\subseteq (0,+\infty)$ satisfying the following conditions:
\begin{itemize}
\item[(i)] $B(x_i,r_i)\cap B(x_j,r_j)=\emptyset$ for $i\neq j$;
\item[(ii)] $U\subseteq\bigcup_{i\in\N} B(x_i,5r_i)$;
\item[(iii)] $\sum_{i=0}^\infty h(B(x_i,5r_i))=\sum_{i=0}^\infty \frac{1}{5r_i} \int_{B(x_i,5r_i)} w\leq C_B P_w(U)$.
\end{itemize}
\end{theorem}

We are now ready to study removable singularities of divergence-free vector fields
\subsection{Removable singularities}

Assume in this whole section that $w$ is a weight. Given vector field $v\in L^\infty_{1/w}(\R^n,\R^n)$ and $\varphi\in \Lip_c(\R^n)$, it is clear that we have, a.e. on $\R^n$:
$$
|v\cdot\nabla\varphi|\leq \left|\frac{v}{w}\right||\nabla\varphi|w\chi_{\supp \varphi}\leq \|v\|_{\infty,1/w} \|\nabla\varphi\|_\infty w\chi_{\supp \varphi}.
$$
Since $w$ is locally integrable, one can define the (extended) distributional divergence of $v$ by:
$$
\la \diver v,\varphi\ra:=-\int_{\R^n} v\cdot \nabla\varphi,
$$
for any $\varphi\in \Lip_c(\R^n)$.

\begin{definition}
Let $S\subseteq\R^n$ be compact and $f\in L^1_{\loc}(\R^n)$. 
The set $S$ is said to be \emph{$L^\infty_{1/w}$-removable with respect to the equation $\diver v=f$} in case for any 
$v\in L^\infty_{1/w}(\R^n,\R^n)$, the equality
\begin{equation}\label{eq.def-eff}
\la\diver v,\varphi\ra=\int_{\R^n} f\varphi
\end{equation}
for any $\varphi\in\Lip_c(\R^n)$ with $\supp \varphi \cap S=\emptyset$ (which we shall abbreviate ``$\diver v=f$ outside $S$'') implies that (\ref{eq.def-eff}) also holds for any $\varphi\in \Lip_c(\R^n)$ (which we shall abbreviate ``$\diver v=f$ in $\R^n$'').
\end{definition}

\begin{remark}\label{rmq.def-eq}
Assume that $w$ is an $A_1$ weight.
According to the fact that given $\varphi\in \Lip_c(\R^n)$, one has $\varphi\in L^1_w(\R^n)\cap L^1_{\loc}(\R^n)$ and $\nabla\varphi\in L^1_w(\R^n)\cap L^1_{\loc}(\R^n)$, \cite[Corollary~2.1.5]{Tu} ensures that the above definition remains unchanged in case $\Lip_c(\R^n)$ is replaced by $C^\infty_c(\R^n)$; yet it is not clear that the two definitions are equivalent without this extra assumption on $w$.
\end{remark}

\begin{remark}
Assume that $f\in L^1_{\loc}(\R^n)$ yields a solution $v_0\in L^\infty_{1/w}(\R^n,\R^n)$ to the equation $\diver v=f$, in the sense that $\diver v_0=f$ in $\R^n$. Assume also that $S$ is $L^\infty_{1/w}$-removable for the equation $\diver v=0$.
Fix now $v\in L^\infty_{1/w}(\R^n,\R^n)$ and assume that $\diver v=f$ outside $S$.
 Since it is clear that one has $\diver (v-v_0)=0$ outside $S$, the removability assumption made on $S$ ensures that one has $\diver (v-v_0)=0$ in $\R^n$; we hence get $\diver v=f$ in $\R^n$, and $S$ is $L^\infty_{1/w}$-removable for the equation $\diver v=f$. Conversely, one shows in a similar fashion that any $L^\infty_{1/w}$-removable set for $\diver v=f$ is also removable for $\diver v=0$. Hence we shall assume in the sequel that $f=0$.
\end{remark}

\subsection{A sufficient condition for a set to be $L^\infty_{1/w}$-removable} 
Let us observe first that it suffices, in order to show that a set is removable, that one is able to construct, in any of its neighborhood, an appropriate Lipschitz approximation of its characteristic function.

\begin{lemma}\label{lem.eff-en-p}
Assume that $S$ is compact and that for any $\epsilon>0$, there exists a neighborhood $U$ of $S$ satisfying $\int_U w\leq\epsilon$ together with a Lipschitz function $\chi\in \Lip_c(\R^n)$ satisfying $\supp \chi\subseteq U$, $0\leq\chi\leq 1$ in $U$, $\chi=1$ in a neighborhood of $S$ as well as $\|\nabla\chi\|_{1,w}\leq\epsilon$. Then $S$ is $L^\infty_{1/w}$-removable for the equation $\diver v=0$.
\end{lemma}
\begin{proof}Assume that one has $\la\diver v,\varphi\ra=0$ for any $\varphi\in \Lip_c(\R^n)$ satisfying $\supp\varphi\cap S=\emptyset$. Fix then $\varphi\in\Lip_c(\R^n)$, let $\epsilon>0$ and let $\chi$ be associated to $U$ and $\epsilon$ as in the above assumption. Observe that one has, by hypothesis:
\begin{eqnarray*}
\la \diver v,\varphi\ra&=&\la\diver v,\varphi\chi\ra
\\&\leq & \|v\|_{\infty,1/w}\Big( \|\nabla\varphi\|_\infty \int_U w +\|\varphi\|_\infty \|\nabla\chi\|_{1,w}\Big)
\\&\leq &\epsilon \|v\|_{\infty,1/w} (\|\nabla\varphi\|_\infty+\|\varphi\|_\infty).
\end{eqnarray*}
It then follows that $\la \diver v,\varphi\ra=0$ for $\epsilon>0$ is arbitrary. This establishes that $S$ is $L^\infty_{1/w}$-removable for the equation $\diver v=0$.
\end{proof}

The following proposition will be useful while showing that compact sets with $\calH^h(S)=0$ are removable for the divergence equation.
\begin{proposition}\label{prop.smoothing}
Let $w$ be a $1$-admissible weight and assume that $V'\subset\subset V\subset\subset U\subset\subset\R^n$  are open sets, and assume that $|\partial V|=|\partial V'|=0$. For any $\epsilon>0$, there exists a Lipschitz function $\varphi\in \Lip_c(\R^n)$ satisfying 
$\supp \varphi\subseteq U$, $\varphi=1$ on $V'$ and
$$
\|\nabla\varphi\|_{1,w}\leq \epsilon +2P_w(V).
$$
\end{proposition}
\begin{proof}
Start by choosing $\psi\in \Lip_c(\R^n)$ satisfying $\chi_{V'}\leq\psi\leq\chi_{V}$. Choose, according to the definition of $P_w(V)$, a sequence $(\varphi_k)\subseteq \Lip(\R^n)$ converging in $L^1_{w}(\R^n)$ to $\chi_V$ and such that for any $k\in\N$, one has:
$$
\|\nabla\varphi_k\|_{1,w}\leq \frac{\epsilon}{4} + P_w(V).
$$
Replacing if necessary $\varphi_k$ by $\min[1,\max(\varphi_k,0)]$ (which does not increase the norm of the gradient of $\varphi_k$ on the complement of a negligible set), one can assume that one has $0\leq\varphi_k\leq 1$ for each $k$. Replacing if necessary $\varphi_k$ by $\theta\varphi_k$ for $k\in\N$, where $\theta\in\Lip_c(\R^n)$ satisfies $0\leq\theta\leq 1$, as well as $\theta=1$ on $V$ and $\theta=0$ outside $U$, and observing that one has, for any $k\in\N$:
$$
|\nabla(\theta\varphi_k)|w\leq |\nabla\theta| |\varphi_k-\chi_V|w+ |\nabla\varphi_k|w\leq \|\nabla\theta\|_\infty|\varphi_k-\chi_V|w+ |\nabla\varphi_k|w,
$$
we may also assume that $\supp\varphi_k\subseteq U$ for any $k\in\N$.

Define now, for $k\in\N$:
$$
\tilde{\varphi}_k:=\psi+(1-\psi)\varphi_k.
$$
It is clear that $\tilde{\varphi}_k$ is Lipschitz and has compact support for any $k\in\N$; moreover one computes, for a.e. $x\in\R^n$ (recall that one has $|\partial V'|=|\partial V|=0$):
$$
|\nabla\tilde{\varphi}_k(x)|\leq
\begin{cases}0&\text{if }x\in V',\\ 2|\nabla\varphi_k(x)|+\|\nabla\psi\|_\infty |1-\varphi_k(x)|&\text{if }x\in V\setminus V',\\
|\nabla\varphi_k(x)| &\text{if }x\in \complement V.
\end{cases}
$$
We hence get, a.e. on $\R^n$, for any $k\in\N$:
$$
|\nabla\tilde{\varphi}_k|w\leq 2|\nabla\varphi_k|w+\|\nabla\psi\|_\infty |\chi_V-\varphi_k|w.
$$
Since we have $\varphi_k\to \chi_V$ in $L^1_{w}(\R^n)$, there exists $k\in\N$ such that:
$$
\|\nabla\tilde{\varphi}_k\|_{1,w} \leq \epsilon + 2P_w(V).
$$
We can hence take $\varphi:=\tilde{\varphi}_k$, and the proof is complete.
\end{proof}

The next lemma, taken from \cite[Lemma~6.2]{KKST}, will be of some help. We include its proof for the sake of clarity.
\begin{lemma}\label{lem.tuominen}
Assume that $w$ is a $1$-admissible weight. Given $x\in\R^n$ and $r>0$, there exists $\rho\in [r,2r]$ such that one has:
$$
P_w(B(x,\rho))\leq C h(B(x,\rho)),
$$
where $C>0$ is independent of $x$ and $\rho$.
\end{lemma}

\begin{proof}
Fix $x\in\R^n$ and $r>0$. Define $\varphi\in \Lip_c(\R^n)$ by the formula:
$$
\varphi(y):=\max\left[ 0,\min\left(2-\frac 1r |y-x|, 1\right)\right],
$$
for $y\in\R^n$. It is clear that one has $|\nabla\varphi|\leq \frac 1r \chi_{B(x,2r)\setminus B(x,r)}$ on $\R^n$. Hence we get, using M.~Miranda's coarea formula (Proposition~\ref{prop.coaire}) and Proposition~\ref{prop.ineg-normes}:
\begin{multline*}
\int_0^1 P_w(B(x,(2-t)r))\,dt=\int_\R P_w(\{\varphi>t\})=\|D\varphi\|_w\\\leq  \int_{\R^n} |\nabla\varphi|w\leq \frac 1r \int_{B(x,2r)\setminus B(x,r)} w\leq C_D h(B(x,r)).
\end{multline*}
If we now choose $\rho\in [r,2r]$ such that one has $$P_w(B(x,\rho))\leq \int_0^1 P_w(B(x,(2-t)r))\,dt,$$ we get:
$$
P_w(B(x,\rho))\leq C_D h(B(x,r))\leq 2C_D h(B(x,\rho)),
$$
and the lemma is proved.
\end{proof}

The next proposition is a first step towards showing $\calH^h$-negligible sets are removable.
\begin{proposition}\label{prop.approx-bv}
Assume that $w$ is a $1$-admissible weight. If $S$ is compact and satisfies $\calH^h(S)=0$, then for any $\epsilon>0$ and any neighborhood $U$ of $S$, one can find an open subset $V\subset\subset U$, satisfying $S\subseteq V$ as well as:
$$
P_w(V)\leq\epsilon.
$$
Moreover $V$ can be chosen to be a finite union of balls.
\end{proposition}

\begin{proof}
Fix $\eta>0$.
Let $0<\delta<\frac 14 \dist(S,\complement U)$ and observe that $\calH^h_\delta(S)=0$. Hence there are balls $B(x_j,r_j)$ for some $x_j\in \R^n$, $r_j>0$, $j\in J$ covering $S$ (since $S$ is compact we can assume $J$ to be finite), satisfying $r_j\leq\delta$ for each $j\in J$ and:
$$
\sum_{j\in J} h(B(x_j,r_j))\leq \eta.
$$
According to Lemma~\ref{lem.tuominen}, choose for each $j\in J$ a radius $\rho_j\in [r_j,2r_j]$ for which one has:
$$
P_w(B(x_j,\rho_j))\leq \frac{C}{\rho_j} \int_{B(x_j,\rho_j)}w;
$$
observe that the doubling property of $w$ yields, for any $j\in J$:
$$
P_w(B(x_j,\rho_j))\leq \frac{CC_D}{r_j} \int_{B(x_j,r_j)}w=CC_Dh(B(x_j,r_j)).
$$
Letting $V:=\bigcup_{j\in J} B(x_j,\rho_j)\subset\subset U$, we hence get, using the subadditivity of the weighted perimeter (\cite[Proposition~4.7]{MIRANDA}):
$$
P_w(V)\leq\sum_{j\in J} P_w(B(x_j,\rho_j))\leq CC_D \sum_{j\in J} h(B(x_j,r_j))\leq CC_D\eta,
$$
from which the desired inequality readily follows.
\end{proof}

The previous proposition together with Lemma~\ref{lem.eff-en-p} and Proposition~\ref{prop.smoothing}, yield a sufficient removability condition.
\begin{theorem}\label{thm.suff}
Assume that $w$ is a $1$-admissible weight. Then any compact set $S\subseteq\R^n$ verifying $\calH^h(S)=0$, is $L^\infty_{1/w}$-removable for $\diver v=0$.
\end{theorem}

\begin{proof}According to Lemma~\ref{lem.leb-negl}, the set $S$ is Lebesgue-negligible. So we may fix $\epsilon>0$, and let $U$ be a neighborhood of $S$ for which one has $\int_Uw\leq\epsilon$ (we may moreover assume that $U$ consists of a finite union of balls).
Start, according to Proposition~\ref{prop.approx-bv}, by choosing an open set $V\subset\subset U$ satisfying $P_w(V)\leq\frac{\epsilon}{4}$ as well as $S\subseteq V$ (and consisting of a finite union of balls). Choose, according to Proposition~\ref{prop.smoothing}, a Lipschitz function $\varphi\in \Lip_c(\R^n)$ equal to $1$ in a neighborhood $V'\subset\subset V$ of $S$ and verifying $\|\nabla\chi\|_{1,w}\leq \frac{\epsilon}{2}+2P_w( V)$. One now computes:
$$
\|\nabla\chi\|_{1,w}\leq \frac{\epsilon}{2}+2P_w(V)\leq \epsilon,
$$
and it follows from Lemma~\ref{lem.eff-en-p} that $S$ is $L^\infty_{1/w}$-removable for $\diver v=0$.
\end{proof}

\subsection{A necessary condition for a set to be $L^\infty_{1/w}$-removable} 

We first state the following estimate for measures satisfying a weighted-Frostman condition. It is similar to \cite[Theorem~2.6.3]{Tu}, although we do not require here $w$ to be $A_1$.
\begin{proposition}\label{prop.dual}
Assume that $w$ is a $1$-admissible weight.
Let $\mu$ be a (nonnegative) Radon measure on $\R^n$ and assume that for any $x\in\R^n$ and any $r>0$, we have :
\begin{equation}\label{eq.frostman-mu}
\mu(B(x,r))\leq C h(B(x,r)).
\end{equation}
Then for any $\varphi\in\Lip_c(\R^n)$ we have:
$$
\int_{\R^n} |\varphi|\,d\mu\leq M\|\nabla\varphi\|_{1,w},
$$
where $M>0$ is independent of $\varphi$.
\end{proposition}
\begin{proof}
Fix $t>0$ and let $U_t$ be the bounded open set defined by $U_t:=\{x\in\R^n:|\varphi(x)|>t\}$.
According to the weighted boxing inequality (Theorem~\ref{thm.boxing}), there exists $(x_i)\subseteq U_t$ and $(r_i)\subseteq (0,\infty)$ satisfying conditions (i) to (iii) in Theorem~\ref{thm.boxing}, with $U_t$ instead of $U$.
Write then
$$
\mu(U_t)\leq\sum_{i=0}^\infty \mu(B(x_i,5r_i))\leq  C\sum_{i=0}^\infty \frac{1}{5r_i}\int_{B(x_i,5r_i)}w\leq C C_B P_w(U_t).
$$
Now we have, using Cavalieri's principle and Miranda's coarea formula (Proposition~\ref{prop.coaire}):
$$
\int_{\R^n}|\varphi|\,d\mu=\int_0^\infty \mu(U_t)\,dt\leq CC_B\int_0^\infty P_w(U_t)\,dt=CC_B\|D|\varphi|\|_w.
$$
According to Proposition~\ref{prop.ineg-normes}, we have $$
\|D|\varphi|\|_w\leq  \|\nabla |\varphi|\|_{1,w}=\|\nabla\varphi\|_{1,w},
$$
since one has $|\nabla |\varphi||=|\nabla\varphi|$ almost everywhere. This finishes the proof since one can take $M:=CC_B>0$.
\end{proof}

The following lemma, of functional analytic nature, avoids unnecessary technicalities in the sequel, and has been suggested by J.~Bo\"el.
\begin{lemma}\label{lem.boel}
Let $X$ and $Y$ be two normed spaces, and assume that $T\colon X\to Y$ is linear and \emph{isometric}, \emph{i.e.} that one has $\|T(x)\|_Y=\|x\|_X$ for every $x\in X$. Then its adjoint map $T^*\colon Y^*\to X^*$ is surjective.
\end{lemma}

\begin{remark}
Observe that one does not need, in the above statement, any kind of completeness to be satisfied, neither by $X$ nor by $Y$.
\end{remark}

\begin{proof}
Fix $f\in X^*$. It is clear by assumption that $T$ is injective; hence the formula:
$$
\la g_0,T(x)\ra:=\la f, x\ra,
$$
defines a linear map $g_0\colon T(X)\to\R$ satisfying 
$$
|\la g_0, T(x)\ra|\leq \|f\| \|x\|_X=\|f\|\|T(x)\|_Y$$ 
for any $x\in X$.
The Hahn-Banach Theorem hence ensures the existence of $g\in Y^*$ verifying $g\upharpoonright T(X)=g_0$, meaning that one has:
$$
\la T^*g,x\ra:=\la g,T(x)\ra=\la g_0,T(x)\ra=\la f,x\ra,
$$
for any $x\in X$. This shows that $T^*g=f$ and establishes the surjectivity of $T^*$.
\end{proof}
We are now able to show that any Radon measure satisfying a weighted Frostman condition is the divergence of a vector field in $L^\infty_{1/w}$.
\begin{theorem}\label{thm.sol}
Let $w$ be a $1$-admissible weight.
If $\mu$ is a (nonnegative) Radon measure verifying condition (\ref{eq.frostman-mu}) above, then there exists $v\in L^\infty_{1/w}(\R^n,\R^n)$ such that, for any $\varphi\in\Lip_c(\R^n)$, one has:
$$
\la \diver v,\varphi\ra=\int_{\R^n}\varphi\, d\mu;
$$
in particular, $\diver v=\mu$ holds on $\R^n$, in the distributional sense.
\end{theorem}

\begin{proof}
It is inspired by \cite[Theorem 3.3]{PT}. Denote by $X$ the space $\Lip_c(\R^n)$ endowed with the norm $\|\varphi\|_X:=\|\nabla\varphi\|_{1,w}$, and define an operator
$$
T\colon X\to L^1_w(\R^n,\R^n), u\mapsto -\nabla u.
$$
Since $T$ is clearly isometric, its adjoint operator
$$
\diver\colon [L^1_w(\R^n,\R^n)]^*=L^\infty_{1/w}(\R^n,\R^n)\to X^*
$$
is surjective. Yet Proposition~\ref{prop.dual} ensures that we have $\mu\in X^*$. The proof is complete.
\end{proof}
\begin{remark}
To obtain the isometric isomorphism between $ [L^1_w(\R^n,\R^n)]^*$ and $L^\infty_{1/w}(\R^n,\R^n)$, it suffices to notice that given $f\in [L^1_w(\R^n,\R^n)]^*$, the formula $v\mapsto \la f,\frac vw \ra$ defines a bounded linear map on $L^1(\R^n,\R^n)$ with the same norm as $f$. Hence there exists $g\in L^\infty(\R^n,\R^n)$ with $\left\Vert g\right\Vert_{\infty}=\left\Vert f\right\Vert$ such that for any $u\in L^1_w(\R^n,\R^n)$ one has
$$
f(u)=\int_{\R^n} u \cdot gw.
$$
Yet the function $h:=gw$ belongs to $L^\infty_{1/w}(\R^n,\R^n)$ and $\left\Vert h\right\Vert_{L^{\infty}_{1/w}}=\left\Vert f\right\Vert$.  
\end{remark}

We now have a necessary condition on a compact set $S$ for it to be $L^\infty_{1/w}$-removable.
\begin{theorem}\label{thm.nec}
Assume that $w$ is a $1$-admissible weight that satisfies condition \eqref{eq.croiss}. 
If the compact set $S\subseteq\R^n$ is $L^\infty_{1/w}$-removable for $\diver v=0$, then $\calH^h_\infty(S)=0$.
\end{theorem}
\begin{proof}
To show this, assume that one has $\calH^h_\infty(S)>0$. According to the above weighted version of Frostman's lemma (Lemma~\ref{lem.frostman}), there exists a nontrivial Radon measure $\mu$ on $\R^n$ supported in $S$ and satisfying:
$$
\mu(B(x,r))\leq C h(B(x,r)),
$$
for any $x\in\R^n$ and $r>0$. Since Theorem~\ref{thm.sol} ensures the existence of a vector field $v\in L^\infty_{1/w}(\R^n,\R^n)$ such that $\diver v=\mu$, we get $\diver v\neq 0$ while $\la\diver v,\varphi\ra=0$ holds in case $\varphi\in\Lip_c(\R^n)$ satisfies $\supp\varphi \cap S=\emptyset$. Hence $S$ cannot be $L^\infty_{1/w}$-removable for $\diver v=0$.
\end{proof}

According to Lemma~\ref{lem.equiv}, Theorems~\ref{thm.suff} and \ref{thm.nec} give a complete characterization of $L^\infty_{1/w}$-removable (compact) subsets of $\R^n$ for the equation $\diver v=0$, in case $w$ satisfies (\ref{eq.croiss}).
\begin{theorem}\label{thm.iff infty}
Assume that $w$ is a $1$-admissible weight and that $w$ satisfies condition \eqref{eq.croiss}.
Let $S\subseteq\R^n$ be compact. Under those assumptions, $S$ is $L^\infty_{1/w}$-removable for $\diver v=0$ if and only if one has $\calH^h(S)=0$.
\end{theorem}

\begin{remark} \label{remcapacity}
The reader may wonder why we did not use the capacity theory to establish Theorem \ref{thm.iff infty}. Indeed, recall that in the classical unweighted case, $\calH^{n-1}$, Hausdoff measure of dimension $n-1$ and $1$-capacity $\capp_1$ (defined by \eqref{p capacity} below) have same zero sets, see \cite[Theorem 3, p. 193]{EG}. A corresponding result holds in the setting of metric spaces: by \cite[Theorem 3.5]{KINNUNEN-ET-AL} Hausdorff content of codimension $1$ and $1$-capacity are comparable for compact sets. The reason why we cannot use this result here is that, in \cite{KINNUNEN-ET-AL}, admissible functions for $1$-capacity belong to the Sobolev space {\it defined using weak upper gradients}, and it is not known if $|\grad u|$ is a $1$-minimal weak upper gradient of each locally Lipschitz function $u$, see \cite[Appendix A2]{BB}. 
\end{remark}
Before discussion on $L^p$-analogues of the previous results, let us mention an interesting example.
\begin{example}\label{ex.cantor}
Given a cube $Q$, we denote by $\sigma(Q)$ the length of its edges.
For $0<\lambda<1/2$, we let $\calE(Q,\lambda)$ stand for the collection
of $2^n$ cubes contained in $Q$ whose edges have length $\lambda\sigma(Q)$, arranged in such a way that each
cube of $\calE(Q,\lambda)$ has a common vertex with $Q$.

Fix a sequence $\lambda=(\lambda_k)_{k\geq 1}\subseteq (0,1/2)$.
Write $E_0:=[0,1]^n$ and $\calE_0=\{E_0\}$. For each $k\geq 1$, define inductively
$$\calE_k(\lambda_1,\dots,\lambda_k):=\cup\{\calE(Q,\lambda_k):Q\in\calE_{k-1}(\lambda_1,\dots,\lambda_{k-1})\}$$ and
$E_k(\lambda_1,\dots,\lambda_k):=\cup\calE_k(\lambda_1,\dots,\lambda_k)$ (or briefly
$\calE_k$ and $E_k$ when the underlying sequence is clear). In particular,
each cube in $\calE_k$ has edges of length $\sigma_k:=\lambda_1\dots\lambda_k$.
We finally let 
$$
E(\lambda) = \bigcap_{k= 1}^{\infty} E_k.
$$

Associated to $\lambda$, we choose a nondecreasing function $h\colon [0,\infty)\to [0,\infty)$ satisfying $h(0)=0$ and $h(\sigma_k)=2^{-kn}$ for each $k\geq 1$.
Recall that the Hausdorff $h$-measure of $E(\lambda)$ is positive
and finite (see \cite[Section 4.9]{MATTILA}).

If for example we fix a real number $0<s<n$ and take $\lambda_k=2^{-\frac ns}$ for each $k$, this yields $\sigma_k=2^{-\frac{kn}{s}}$ for each $k$, and we may take $h(t)=t^s$. Hence in this case the set $E:=E(\lambda)$ has Hausdorff dimension $s$ and is $s$-Ahlfors regular according to \cite[Theorem~8.3.2]{MaTy} (see also \cite{Mo} and \cite{Hu}). It hence follows from \cite[Theorem~7]{ACDT} that for any $0\leq\gamma <1$, the map $w_\gamma$ defined for a.e. $x\in\R^n$ by 
$$
w_\gamma(x):=\dist(x,C)^{\gamma(s-n)},
$$ 
is an $A_1$-weight.

Observing now that $E_k$ consists of $2^{kn}$ cubes of side length $\sigma_k$, compute, for any one of those cubes $Q$ and $B\supseteq Q$, the smallest ball containing $Q$ having the same center:
$$
\frac{1}{|B|}\int_B w_\gamma\leq A_\gamma \infess_B w_\gamma,
$$
where $A_\gamma>0$ is a constant associated with the $A_1$-property of $w_\gamma$. Since the set:
$$
\Big\{x\in B: \dist(x,E)\geq \frac 14 \left(\frac{\sigma_k}{2}-\sigma_{k+1}\right)\Big\}
$$
has positive Lebesgue measure, this yields:
$$
\frac{1}{|B|}\int_B w_\gamma\leq A_\gamma 4^{\gamma(n-s)} \left(\frac{\sigma_k}{2}-\sigma_{k+1}\right)^{\gamma(s-n)}.
$$
Since $2^{kn}$ of those balls suffice to cover $E_k$, we compute:
\begin{align*}
\calH^h(E_k)
&\leq 2^{kn} h(B)=2^{kn} \cdot \frac{\sqrt{2}}{\sqrt{n}\sigma_k} \int_B w_\gamma\\
&\leq C(n,\gamma, s) 2^{kn}  \sigma_k^{n-1} \left(\frac{\sigma_k}{2}-\sigma_{k+1}\right)^{\gamma(s-n)}\\
&= C(n,\gamma,s) 2^{kn} \sigma_k^{n-1+\gamma(s-n)} \left(\frac 12 -2^{-\frac ns}\right)^{\gamma(n-s)}\\
&=C'(n,\gamma,s)2^{-k\left[(1-\gamma)\frac{n^2}{s}-\frac ns +(\gamma -1)n\right]}.
\end{align*}
The last expression tends to zero for $k\to\infty$, provided that one has:
\begin{equation}\label{small s}
s<n-\frac{1}{1-\gamma}.
\end{equation}
This yields $\calH^h(E)=0$ if inequality \eqref{small s} is satisfied. In particular, the set $E$ is then $L^\infty_{1/w_\gamma}$-removable for the equation $\diver v=0$ while we have $w_\gamma=+\infty$ on $E$.
\end{example}

\section{The case of weighted $L^p$ vector fields}\label{sec.lp}
In this section, we study the removability question for the equation $\diver v=0$ for $L^p_{1/w}$-vector fields, that is measurable functions $v:\R^n\rightarrow \R^n$ such that 
$$
\left\Vert v\right\Vert_{L^p_{1/w}}:=\left(\int_{\R^n} \left\vert v(x)\right\vert^p\frac 1{w(x)}dx\right)^{1/p}<+\infty,
$$
$1<p<+\infty$.
In Theorems \ref{cap implies rem} and \ref{rem implies cap}, we give a characterization for removability by showing that, under suitable assumptions on the weight $w$, a compact set $S$ is $L^p_{1/w}$-removable for $\diver v=0$ if and only if $Cap_{p'}^{w^{p'-1}}(S)=0$. Before that, we recall the definition of weighted Sobolev spaces, different capacities and prove some lemmas. 
In the whole section, $p$ and $p'$ are conjugate exponents satisfying $1/p+1/p'=1$.

\begin{remark}
It follows from the H\"older inequality that if $w\in L^{p'-1}_{\loc}(\R^n)$, then the distributional divergence of $v\in L^p_{1/w}(\R^n,\R^n)$ is well defined by the formula:
$$
\la\diver v,\varphi\ra:=-\int_{\R^n} v\cdot\nabla\varphi,
$$
for $\varphi\in\Lip_c(\R^n)$.
\end{remark}

Removability in the $L^p$-case is defined similarly as for $L^{\infty}$-vector fields.
\begin{definition}
Let $S\subseteq\R^n$ be compact and $1<p<\infty$. Let $w$ be a weight.
The set $S$ is \emph{$L^p_{w}$-removable with for the equation $\diver v=0$} (in the distributional sense) if for any 
$v\in L^p_{w}(\R^n,\R^n)$, the validity of the equality
\begin{equation}\label{eq.def-eff p}
\la\diver v,\varphi\ra=0
\end{equation}
for any $\varphi\in\Lip_c(\R^n)$ with $\supp \varphi \cap S=\emptyset$ implies that (\ref{eq.def-eff p}) also holds for any $\varphi\in \Lip_c(\R^n)$.
\end{definition}
\begin{remark}\label{rmq.def-eq2}
Given $1<p<\infty$ and an $A_p$ weight $w$, we can again (as in Remark~\ref{rmq.def-eq}) rely on \cite[Corollary~2.1.5]{Tu} (see \cite[Lemma~2.4]{Mi} for a proof in this case) to claim that the above definition remains unchanged in case $\Lip_c(\R^n)$ is replaced by $C^\infty_c(\R^n)$.
\end{remark}

\subsection{Weighted Sobolev spaces}
There are two ways to define weighted Sobolev spaces: using distributional derivatives or the closure of smooth compactly supported functions in the weighted Sobolev norm. 
The space $W^{1,p}_w(\R^n)$ consists of functions $u\in L^p_w(\R^n)$ whose distributional derivatives of order one belong to $L^p_w(\R^n)$. The space is equipped with norm
\[
\|u\|_{W^{1,p}_w(\R^n)}:=\|u\|_{p,w}+\|\grad u\|_{p,w},
\]
where $\grad u$ is the distributional gradient of $u$.
The space ${H^{1,p}_w(\R^n)}$ is the closure of $C^{\infty}_c(\R^n)$ under the norm $\|\cdot\|_{W^{1,p}_w(\R^n)}$.
For $A_p$-weights, $1<p<\infty$, $H$ and $W$ definitions of weighted Sobolev spaces give the same space by \cite[Theorem 2.5]{Ki}.

\subsection{Potentials and capacities}

We need three different capacities: Sobolev capacity and the weighted versions of Riesz and Bessel capacity.
The capacity theory in weighted Sobolev spaces has been developed to study nonlinear potential theory, see \cite{HeKiMa}, \cite{A2} and \cite{Tu}. 
For the properties of Riesz and Bessel potentials, we refer, in the classical case to \cite{AH}, \cite[Chapter V]{St} and \cite{Z} and in the weighted case to \cite{A2}, \cite{Ai}, \cite{Tu}.
As in the unweighted Euclidean space, Riesz and Bessel potentials and the corresponding capacities are closely related to Sobolev spaces.

\begin{definition}\label{Definition: Sobo capacity}
Let $w$ be a weight and let $1\le p<\infty$. The {\em weighted Sobolev $p$-capacity} of a set $E$ is 
\[
\Capa_{p}^{w}(E)
:=\inf\left\{\int_{\rn}|\ph|^p w + \int_{\rn}|\grad{\ph}|^p w\right\},
\]
where the infimum is taken over all functions $\ph\in\cA(E)$, where we let:
\[
\cA(E):=\big\{\ph\in C^\infty_c(\rn):\ph\ge1 \text{ in a neighbourhood of } E \big\},
\]
where as usual $C^\infty_c(\rn)$ denotes the space of all compactly supported smooth functions in $\R^n$.
\end{definition}

\begin{remark} \label{testfunction}
Assume that $E$ is compact and that $w\in A_p$. We claim that
$$
\Capa_{p}^{w}(E)
=\inf\left\{\int_{\rn}|\ph|^p w + \int_{\rn}|\grad{\ph}|^p w\right\},
$$
where the infimum is taken over all functions $\ph\in\cB(E)$, defined by
\[
\cB(E):=\big\{\ph\in C^\infty_c(\rn): 0\leq \ph\leq 1\text { in } \R^n \text{ and }\ph=1 \text{ in a neighbourhood of } E \big\}.
\]
Indeed, let $\ph\in \cA(E)$. Define $\psi:=\max(0,\inf(\ph,1))$. The function $\psi$ belongs to $\Lip_c(\R^n)$ with $\left\vert \nabla \psi\right\vert\leq \left\vert \nabla \varphi\right\vert$ almost everywhere, satisfies $0\leq \psi\leq 1$ in $\R^n$ and $\psi=1$ in a neighbourhood of $E$ and
$$
\int_{\rn}|\psi|^p w + \int_{\rn}|\grad{\psi}|^p w\leq \int_{\rn}|\ph|^p w + \int_{\rn}|\grad{\ph}|^p w.
$$
According to \cite[Corollary 2.1.5]{Tu} (see also \cite[Lemma~2.4]{Mi} for the proof in our case where $p>1$) and since $\psi$ is compactly supported, for all $\varepsilon>0$, there exists a function $\psi_{\varepsilon}$ belonging to $C^{\infty}_c(\R^n)$ (obtained by mollification of $\psi$ with a suitable convolution kernel), satisfying $0\leq\psi_{\varepsilon}\leq 1$ in $\R^n$ and $\psi_{\varepsilon}=1$ in a neighbourhood of $E$, such that 
$$
\|\psi_\epsilon\|_{p,w}^p\leq \epsilon+\|\psi\|_{p,w}^p\quad\text{and}\quad\|\nabla\psi_\epsilon\|_{p,w}^p\leq \epsilon+\|\nabla\psi\|_{p,w}^p,
$$
which proves the claim, for we hence get:
$$
\int_{\rn}|\psi_\epsilon|^p w + \int_{\rn}|\grad{\psi_\epsilon}|^p w\leq 2\epsilon+ \int_{\rn}|\ph|^p w + \int_{\rn}|\grad{\ph}|^p w.
$$
\end{remark}
\begin{remark}
Assume that $1<p<n$ is a real number and that $w$ is a $p$-admissible weight.
According to \cite[Corollary~2.39]{HeKiMa}, the weighted Sobolev capacity $\Capa_{p}^w$ has the same zero sets as the capacity $\capa_p^w$ defined by
\begin{equation}\label{p capacity}
\capa_{p}^{w}(E)
:=\inf\left\{\int_{\rn}|\grad{\ph}|^p w\right\},
\end{equation}
where the infimum is taken over all functions $\ph\in \cA(E)$.
\end{remark}

\subsubsection *{Riesz potential and capacity} 
Let $0<R\leq\infty$. The ($R$-truncated) Riesz potential (of order $1$) 
of a nonnegative, locally integrable function $u$ is the function $I_{1,R} u\colon\rn\to[0,\infty]$ given by 
\[
I_{1,R} u(x):=\int_{B(x,R)}\frac{u(y)}{|x-y|^{n-1}}\,dy,
\]
and the ($R$-truncated) Riesz potential of a measure $\mu\in\cM_+(\rn)$ is $I_{1,R}\mu\colon\rn\to[0,\infty]$,  given by
\[
I_{1,R}\mu(x):=\int_{B(x,R)}\frac1{|x-y|^{n-1}}\,d\mu(y).
\]
Taking $R=\infty$ gives the usual Riesz potential as defined in \cite{AH} and \cite{Tu} (note that usually the integrals above are multiplied by a constant $\gamma(n)$; since we are not interested in the exact values of potentials and capacities, we omit the constant).

\begin{definition}\label{Riesz cap}
Let $w$ be a weight and let $1<p<\infty$.  
For $0<R\leq\infty$, we define the {\em ($R$-truncated) weighted Riesz capacity }of a set $E\subset\rn$ by
\[
R^w_{1,p;R}(E)
:=\inf\Big\{\|f\|^p_{p,w}:f\ge0,\, I_{1,R}f\ge1\text{ on }E\Big\},
\] 
and we let $R^w_{1,p}(E):=R^w_{1,p;\infty}(E)$.
\end{definition}

If $I_{1,R}\mu\in L^p_w(\rn)$, then we say that measure $\mu$ has {\em a finite weighted $(p,w;R)$-energy} (or finite weighted $(p,w)$-energy in case $R=\infty$).

\subsubsection *{Bessel potential and capacity} 
The Bessel kernel $G_1$ is the tempered distribution whose Fourier transform is $\hat G_1(x)=(1+|x|^2)^{-1/2}$. 
It is actually a function with the same singularity at $0$ as the Riesz kernel $I_1(x)=|x|^{1-n}$, but has more rapid decay at infinity. 
Similarly as for Riesz potentials, the Bessel potential of a measurable function $u$ is convolution $\cG_1u=G_1*u$ and 
of measure $\mu\in\cM_+(\rn)$, $\cG_1\mu=G_1*\mu$. 

Riesz potentials are easier to handle but Bessel potentials have better mapping properties. 
The most important properties for us are the inequality
\begin{equation}\label{Riesz bigger}
0<G_1(x)<CI_1(x) 
\end{equation}
for all $x\in\re$, the fact that
$$
G_1(x)=I_1(x) + o(I_1(x))
$$
as $x\to 0$ (for these two facts, see \cite[Section 3.1.2]{Tu}),
and the following Theorem from \cite[Theorem 3.3]{Mi}, \cite{Ni} which says that the Bessel potential space equals the weighted Sobolev space for $A_p$-weights.

\begin{theorem}\label{bessel equals sobo}
Let $1<p<\infty$. Let $w$ be an $A_p$-weight. A function $u$ belongs to $H^{1,p}_w(\rn)$ if and only if there is a function $f\in L^p_w(\rn)$ such that $u=\cG_1f$. Moreover, there is a constant $C=C(n,p,w)>0$ such that 
\[
\frac1C\|f\|_{p,w}\le\|u\|_{H^{1,p}_w(\rn)}\le C\|f\|_{p,w}.
\]
\end{theorem}

\begin{definition}\label{bessel cap}
Let $1<p<\infty$. Let $w$ be a weight. {\em The weighted Bessel capacity }of a set $E\subset\rn$ is
\[
B^w_{1,p}(E)
=\inf\Big\{\|f\|^p_{p,w}:f\ge0,\, \cG_1f\ge1\text{ on }E\Big\}.
\] 
\end{definition}

\subsection{Connection between different capacities}
In the proof of Theorem \ref{rem implies cap}, we need the property that weighted Sobolev, Bessel and Riesz capacities have same zero sets. 
Connection between Sobolev and Bessel capacities is much stronger; Theorem \ref{bessel equals sobo} implies that the weighted Bessel and Sobolev capacities for an $A_p$-weight are comparable. For a proof, see \cite[Theorem 3.5.2]{Tu}.

\begin{theorem}\label{bessel and sobo cap}
Let $1<p<\infty$. Let $w$ be an $A_p$-weight. There is a constant $C>0$ such that 
\[
\frac1C \Capa_p^w(K)\le B^w_{1,p}(K)\le C\,\Capa_p^w(K)
\]
for all compact sets $K\subset\rn$.
\end{theorem}

Concerning Riesz capacities, we have, by \eqref{Riesz bigger}, that
$$
R^w_{1,p}(E)\le C B^w_{1,p}(E)
$$
for all sets $E\subset\rn$. 
In the unweighted case, Bessel and Riesz capacities have same sets of finite capacity by \cite[Theorem 1]{A1}. 
This Theorem says that if $p<n$, then 
\[
B_{1,p}(E)\le C\big(R_{1,p}(E)+R_{1,p}(E)^{n/(n-p)}\big).
\]
In the weighted case, Bessel and Riesz capacities are locally comparable if the $A_p$-weight satisfies an additional integrability condition, see \cite[Lemma 20]{Ai}. 
The following lemma (resulting from \cite[Theorem~3.3.7 and Lemma~3.3.8]{Tu}) together with Theorem \ref{bessel and sobo cap} implies that for $A_p$-weights, weighted Sobolev, Bessel and (localized) Riesz capacities are comparable.

\begin{lemma}\label{zero sets}
Let $1<p<\infty$ and $R>0$ be real numbers. Let $w$ be an $A_p$-weight. 
 There is a constant $C_R>0$ such that 
\[
\frac{1}{C_R} R^w_{1,p;R}(K)\le B^w_{1,p}(K)\le C_R\,R^w_{1,p;R}(K),
\]
for all compact sets $K\subset\rn$.
\end{lemma}

\begin{remark}\label{measure with finite energy}
We could use the general theory for $L^q$-capacities and dual definitions of capacity from \cite[Chapter 2]{AH} for weighted Riesz capacities, see also \cite{A2} and \cite[Chapter 3]{Tu}. We do not repeat the details of the theory because  the only property we need is that for each (compact) set $S\subseteq B(0,R)\subseteq \rn$ of positive weighted $(3R)$-localized Riesz capacity, there is a nonnegative non-zero measure supported in $S$ with finite energy.

Indeed, by \cite[Theorem 2.5.3]{AH}, for each compact set $K\subset\rn$, there is a measure $\mu_K\in\cM_+(K)$ such that
\begin{equation}\label{general cap}
\Capa_{g,q}(K)=\mu_K(K)=\int_{\rn}\Big(\int_{\rn}g(x,y)\,d\mu_K(x)\Big)^{q'}\,d\nu(y)
\end{equation}
where $1<q<\infty$, $1/q+1/q'=1$, capacity $\Capa_{g,q}$ is defined using a kernel function $g(x,y)$ and $\nu$ is a positive measure in $\rn$. 

Let $S\subseteq B(0,R)$ be a compact set with $Cap_{p'}^{w^{p'-1}}(S)>0$. A non-zero measure supported in $S$ with finite $(p,1/w;3R)$ energy is obtained using \eqref{general cap} and selecting $q'=p$, because for
\[
d\nu=w^{p'-1}\,dy\quad\text{ and }\quad g(x,y)=\frac{\ch{B(0,3R)}(x-y)}{|x-y|^{n-1}}w(y)^{1-p'},
\]
we have
\begin{eqnarray*}
\lefteqn{\int_{\rn}\Big(\int_{\rn}g(x,y)\,d\mu(x)\Big)^{p}\,d\nu(y)}&&\\
&=&\int_{\rn}\Big(\int_{\rn}\frac{\ch{B(0,3R)}(x-y)}{|x-y|^{n-1}}w(y)^{1-p'}\,d\mu(x)\Big)^{p}w(y)^{p'-1}\,dy\\
&=&\int_{\rn}\Big(\int_{\rn}\frac{\ch{B(0,3R)}(x-y)}{|x-y|^{n-1}}\,d\mu(x)\Big)^{p}w(y)^{(1-p')p-1+p'}\,dy\\
&=&\int_{\rn}\big(I_{1,3R}\mu\big)^{p}w(y)^{-1}\,dy.
\end{eqnarray*}
\end{remark}

\subsection{Removable singularities}\label{section: removability}
The following result is a counterpart of  Proposition~\ref{prop.dual} in the $L^p$-case. In the unweighted case, compare it with \cite[Theorem 3.2]{PT}.

\begin{proposition}\label{finite energy lemma}
Let $w$ be a weight and let $1<p<\infty$ and $R>0$ be real numbers. Assume that $w^{p'-1}$ is $p'$-admissible.
If $\mu\in\cM_+(B(0,R))$ has a finite $(p,1/w;3R)$-energy, then 
\[
\int_{\rn}|\ph|\,d\mu\le C(n) \|I_{1,3R}\mu\|_{p,w^{-1}}\|\grad \varphi\|_{p',w^{p'-1}}
\]
for each $\ph\in \Lip_c(\R^n)$.
\end{proposition}

\begin{proof}
Assume first that $\varphi\in \Lip_c(\R^n)$ satisfies $\supp\varphi\subseteq B(0,2R)$.
Fix $x\in B(0,R)$. Since for any $y\in B(0,R)$ one has $x-y\in B(0,3R)$, and since, according to \cite[(18) p.~125]{St}, there is a constant $C=C(n)>0$ such that, for all $x\in \R^n$, 
\[
\ph(x)=C(n)\int_{\rn}\frac{\grad \ph(y)\cdot(x-y)}{|x-y|^n}\,dy,
\]
we get, using the fact that $\varphi$ vanishes outside $B(0,2R)$ that, for all $x\in \R^n$,
$$
\ph(x)=C(n)\int_{\rn}\frac{\grad \ph(y)\cdot(x-y)}{|x-y|^n}\ch{B(0,3R)}(x-y)\,dy.
$$
Hence we have, using the Fubini Theorem and the H\"older inequality,
\begin{align*}
\int_{\rn}|\ph|\,d\mu
&\le C(n) \int_{\rn}I_{1,3R}|\grad \ph|\,d\mu
= C(n)\int_{\rn} |\grad \ph| I_{1,3R}\mu\\
&\le C(n)\Big(\int_{\rn}|\grad \ph|^{p'}w^{p'-1}\Big)^{1/p'}\Big(\int_{\rn}(I_{1,3R}\mu)^{p}{w^{-1}}\Big)^{1/p},
\end{align*}
and the claim follows in case one has $\supp\varphi\subseteq B(0,2R)$.

Given a general $\varphi\in \Lip_c(\R^n)$, choose $\chi\in\Lip_c(\R^n)$ satisfying $\ch{B(0,R)}\leq\chi\leq\ch{B(0,2R)}$ as well as $\|\nabla\varphi\|_\infty\leq2/R$. According to the computations before and to the fact that $\supp\mu\subseteq B(0,R)$, we know that one has:
$$
\int_{\R^n} |\varphi|\,d\mu
=\int_{\R^n} |\varphi\chi|\,d\mu\leq C(n)\|I_{1,3R}\mu\|_{p,w^{-1}} \|\nabla(\varphi\chi)\|_{p',{w^{p'-1}}}.
$$
Since we have $\nabla(\varphi\chi)=\chi\nabla\varphi+\varphi\nabla\chi$ a.e. on $\R^n$ and since $\chi=0$ outside $B(0,2R)$, we get:
$$
 \|\nabla(\varphi\chi)\|_{p',w^{p'-1}}\leq \|\nabla\varphi\|_{p',w^{p'-1}} +\frac 2R \bigg(\int_{B(0,2R)} |\varphi|^{p'} w^{p'-1}\bigg)^{1/p'}.
 $$
Yet $w^{p'-1}$ being $p'$-admissible, we get:
$$
\int_{B(0,2R)} |\varphi|^{p'} w^{p'-1}\leq \kappa (2R)^{p'}  \int_{\R^n} |\nabla\varphi|^{p'} w^{p'-1},
$$
where $\kappa$ is a constant associated to the $p'$-admissibility of the weight $w^{p'-1}$ according to Poincar\'e inequality (\ref{eq.poincare'}). This finally gives rise to the following inequality:
$$
 \|\nabla(\varphi\chi)\|_{p',w^{p'-1}}\leq (1+4\kappa^{1/p'}) \|\nabla\varphi\|_{p',w^{p'-1}}
 $$
 and the proof is complete.
 \end{proof}
\begin{remark}We have assumed that $w^{p'-1}\in A_{p'}$. By the definition of conjugate exponents $1/p+1/p'=1$ and $A_q$-weights, this means that 
\[
\sup_B\Big(\,\vint{B}w^{p'-1}\Big)\Big(\,\vint{B}w^{-1}\Big)^{1/(p'-1)}<\infty
\]
and hence that $1/w$ is locally integrable on $\R^n$. 
\end{remark}

A counterpart of Theorem~\ref{thm.sol} in the $L^p$-case says that any nonnegative Radon measure with finite $(p,1/w)$-energy is the divergence of a vector field in $L^{p}_{1/w}$. 

\begin{theorem}\label{finite energy implies solution}
Let $w\in L^{p'-1}_{\loc}(\R^n)$ be a weight and let $R>0$ and $1<p<\infty$ be real numbers. Assume that $w^{p'-1}$ is $p'$-admissible.
If $\mu\in\cM_+(B(0,R))$ is a measure with finite $(p,1/w; 3R)$-energy, then there exists $v\in L^p_{1/w}(\rn,\rn)$ such that  
$\diver v=\mu$ in $\rn$ in the distributional sense.
\end{theorem}

\begin{proof}
Let $X:=\Lip_c(\R^n)$ be endowed with the norm 
\[
\|\ph\|_X=\|\grad \ph\|_{p',w^{p'-1}}
\]
(notice that it is clear that this norm is well defined),
and let $T\colon X\to L^{p'}_{w^{p'-1}}(\rn,\R^n)$, $Tu=-\grad u$. 
Since $T$ is a linear isometry, Lemma~\ref{lem.boel} implies that the adjoint operator 
\[
T^*\colon  \big(L^{p'}_{w^{p'-1}}(\rn,\R^n)\big)^*\to X^*
\]
is surjective. A standard argument of functional analysis recalled below shows that 
$\big(L^{p'}_{w^{p'-1}}(\rn,\R^n)\big)^*=L^p_{1/w}(\rn,\R^n)$. Since
\[
\langle T^*v,\ph\rangle=\langle v,T\ph\rangle
=-\int_{\rn}v\cdot\grad \ph
\]
for all $\ph\in X$, we have that $T^*=\diver$ in the distributional sense. 
The claim follows because $\mu\in X^*$ by Proposition \ref{finite energy lemma}.
\end{proof}

\begin{remark}
To obtain the isometric isomorphism $ [L^{p'}_{w^{p'-1}}(\R^n,\R^n)]^*=L^p_{1/w}(\R^n,\R^n)$, it suffices to notice that given $f\in [L^{p'}_{w^{p'-1}}(\R^n,\R^n)]^*$, the formula $v\mapsto \la f,\frac {v}{w^{1/p}} \ra$ defines a bounded linear map on $L^{p'}(\R^n,\R^n)$ with the same norm as $f$. Hence there exists $g\in L^p(\R^n,\R^n)$ with $\left\Vert g\right\Vert_{L^p(\R^n)}=\left\Vert f\right\Vert$ such that for any $u\in L^{p'}_{w^{p'-1}}(\R^n,\R^n)$ one has
$$
\la f,u\ra=\int_{\R^n} u \cdot gw^{1/p}.
$$
Yet the function $h:=gw^{1/p}$ belongs to $L^p_{1/w}(\R^n,\R^n)$ and $\left\Vert h\right\Vert_{L^p_{1/w}}=\left\Vert f\right\Vert$.
\end{remark}

\begin{theorem}\label{cap implies rem}
Let $1<p<\infty$ be a real number. Let $w\in L^{p'-1}_{\loc}(\R^n)$ be a weight such that $w^{p'-1}\in A_{p'}$.
Let $S\subset\rn$ be a compact set.  
If $Cap_{p'}^{w^{p'-1}}(S)=0$, then $S$ is $L^p_{1/w}$-removable for
\begin{equation}\label{div v=0}
\diver v=0. 
\end{equation}
\end{theorem}

\begin{proof}
Since $Cap_{p'}^{w^{p'-1}}(S)=0$, by Remark \ref{testfunction}, there are functions $u_k\in C^\infty_0(\rn)$, $k\in\n$, such that $0\le u_k\le 1$ and $u_k=1$ {in a neighborhood of $S$} for all $k$, 
$\|\grad u_k\|_{p',w^{p'-1}}\to0$ and $u_k\to0$ as $k\to\infty$ almost everywhere. 
Let $v\in L^p_{1/w}(\R^n,\R^n)$ be a solution of \eqref{div v=0} in $\rn\setminus S$. Let $\ph\in\Lip_c(\rn)$. 
We have to show that   
\[
\int_{\rn}v\cdot\grad\ph=0. 
\]
Since $\diver v=0$ outside $S$ and $(1-u_k)\ph\in\Lip_c(\rn\setminus S)$, we have 
\begin{align*}
\int_{\rn}v\cdot\grad\ph
&=\int_{\rn}v\cdot\grad((1-u_k)\ph) + \int_{\rn}v\cdot\grad(u_k\ph)\\
&= \int_{\rn}v\cdot\grad(u_k\ph).
\end{align*}
Hence, using the H\"older inequality, we obtain
\begin{align*}
\Big|\int_{\rn}v\cdot\grad\ph\Big|
&\le\int_{\rn}|v\cdot\grad(u_k\ph)|\\
&\le\|v\|_{p,w^{-1}}\Big(\|u_k\grad\ph\|_{p',w^{p'-1}} + \|\ph\grad u_k\|_{p',w^{p'-1}}\Big)
\to0
\end{align*}
as $k\to\infty$ by the properties of functions $u_k$ and $\ph$. 
This shows that $S$ is removable for $\diver v=0$.
\end{proof}

\begin{theorem}\label{rem implies cap}
Let $1<p<\infty$ be a real number. 
Let $w\in L^{p'-1}_{\loc}(\R^n)$ be a weight such that $w^{p'-1}\in A_{p'}$.
Let $S\subset\rn$ be a compact set.  
If $S$ is $L^p_{1/w}$-removable for $\diver v=0$, then $Cap_{p'}^{w^{p'-1}}(S)=0$.
\end{theorem}

\begin{proof}
Since $S$ is compact, there exists $R>0$ such that $S\subseteq B(0,R)$.

If $Cap_{p'}^{w^{p'-1}}(S)>0$, then, by Theorem \ref{bessel and sobo cap} and Lemma \ref{zero sets}, $R^{w^{p'-1}}_{1,p';3R}(S)>0$. Hence, by Remark \ref{measure with finite energy}, there is a nonnegative non-zero measure $\mu_S$ supported in $S$ such that:
\[
 \int_{\rn}\big(I_{1,3R}\mu_S\big)^{p}w(y)^{-1}\,dy<+\infty.
\]
Theorem \ref{finite energy implies solution} implies that there exists $v\in L^p_{1/w}(\rn,\rn)$ such that  
$\diver v=\mu_S$ in $\rn$ in the distributional sense. 
This shows that $S$ is not removable for $\diver v=0$. Hence we have that $Cap_{p'}^{w^{p'-1}}(S)=0$.

\end{proof}

In summary, we proved the following result.
\begin{theorem}\label{rem equiv cap}
Let $1<p<\infty$ be a real number. 
Let $w\in L^{p'-1}_{\loc}(\R^n)$ be a weight such that $w^{p'-1}\in A_{p'}$.
Let $S\subset\rn$ be a compact set.  
Under those assumptions, the set $S$ is $L^p_{1/w}$-removable for $\diver v=0$ if and only if one has $Cap_{p'}^{w^{p'-1}}(S)=0$.
\end{theorem}
 
\noindent {\bf Acknowledgements:} 
This research was conducted during the visits of the authors to Laboratoire de Math\'ematiques d'Orsay, the Institut Fourier of the Universit\'e de Grenoble and Department of Mathematics and Statistics of University of Jyv\"askyl\"a and to Forschungsinstitut f\"ur Mathematik of ETH Z\"urich. The authors wish to thank the institutes for the kind hospitality.  The first author was partially  supported by the French ANR project ``GEOMETRYA'' no.~ANR-12-BS01-0014. The second author was partially  supported by the French ANR project ``HAB'' no.~ANR-12-BS01-0013). The third author was supported by the Academy of Finland, grants no.\ 135561 and 272886.

\vspace{0.3cm}
\noindent
\small{\textsc{L.M.},}
\small{\textsc{Laboratoire de Math\'ematiques d'Orsay, Universit\'e Paris-Sud, CNRS UMR8628, Universit\'e Paris-Saclay},}
\small{\textsc{B\^atiment 425},}
\small{\textsc{F-91405 Orsay Cedex},}
\small{\textsc{Fran\-ce}}\\
\footnotesize{\texttt{Laurent.Moonens@math.u-psud.fr}}

\vspace{0.3cm}
\noindent
\small{\textsc{E.R.},}
\small{\textsc{Universit\'e Joseph Fourier, Institut Fourier},}
\small{\textsc{100 rue des Maths, B.P. 74},}
\small{\textsc{F-38402 Saint-Martin-d'H\`eres},}
\small{\textsc{France}}\\
\footnotesize{\texttt{Emmanuel.Russ@ujf-grenoble.fr}}

\vspace{0.3cm}
\noindent
\small{\textsc{H.T.},}
\small{\textsc{Department of Mathematics and Statistics},}
\small{\textsc{P.O. Box 35},}
\small{\textsc{FI-40014 University of Jyv\"askyl\"a},}
\small{\textsc{Finland}}\\
\footnotesize{\texttt{heli.m.tuominen@jyu.fi}}

\end{document}